\documentclass[10pt]{amsart}
\usepackage{amssymb,amsfonts,amsthm,amscd,stmaryrd,dsfont,esint,upgreek}
\usepackage[numbers]{natbib}
\usepackage[mathcal]{euscript}

\usepackage{color} 
\usepackage[usenames,dvipsnames]{xcolor}

\definecolor{orange}{RGB}{255,127,0}

\usepackage{constants}
\newconstantfamily{C}{symbol=C}
\newconstantfamily{c}{symbol=c}
\newconstantfamily{O}{symbol=\Omega}

\theoremstyle{plain}
\newtheorem{thm}{Theorem}
\newtheorem{cor}{Corollary}
\newtheorem{lem}[cor]{Lemma}
\newtheorem{prop}[cor]{Proposition}

\theoremstyle{definition}
\newtheorem{definition}[cor]{Definition}

\numberwithin{cor}{section}
\numberwithin{equation}{section}

\newcommand{\R}{\mathbb{R}}

\renewcommand{\d}{d}
\newcommand{\Rd}{\mathbb R^\d}
\newcommand{\ep}{\varepsilon}

\renewcommand{\P}{\mathbb{P}}
\renewcommand{\O}{\Omega}
\newcommand{\F}{\mathcal{F}}

\renewcommand{\L}{\mathcal{L}}

\newcommand{\Cs}{\mathcal{C}}
\renewcommand{\tilde}{\widetilde}
\renewcommand{\bar}{\overline}

\newcommand{\BS}{\mathcal{U}_0}

\newcommand{\od}[1]{#1^\star}

\newcommand{\Sy}{{\mathbb{S}^\d}}
\newcommand{\Pu}{{\mathcal{P}}}
\newcommand{\m}{m}  

\DeclareMathOperator{\tr}{tr}

\DeclareMathOperator{\dist}{dist}

\DeclareMathOperator*{\esssup}{ess\,sup}
\DeclareMathOperator*{\essinf}{ess\,inf}

\begin{document}

\title[Stochastic homogenization of fully nonlinear equations]{Stochastic homogenization of fully nonlinear uniformly elliptic equations revisited}

\author[S. N. Armstrong]{Scott N. Armstrong}
\address{Department of Mathematics, University of Wisconsin, 
Madison, WI 53706 \\ and Ceremade (UMR CNRS 7534), Universit\'e Paris-Dauphine, Paris, France}
\email{armstrong@ceremade.dauphine.fr}

\author[C. K. Smart]{Charles K. Smart}
\address{Department of Mathematics, Massachusetts Institute of
Technology, Cambridge, MA 02139}
\email{smart@math.mit.edu}

\date{\today}


\begin{abstract}
We give a simplified presentation of the obstacle problem approach to stochastic homogenization for elliptic equations in nondivergence form. Our argument also applies to equations which depend on the gradient of the unknown function. In the latter case, we overcome difficulties caused by a lack of estimates for the first derivatives of approximate correctors by modifying the perturbed test function argument to take advantage of the spreading of the contact set. 
\end{abstract}

\maketitle

\section{Introduction} \label{I}

In this short article we present a simplified proof of the homogenization of nondivergence form uniformly elliptic equations in stationary-ergodic random media and clarify the result for equations with dependence on the gradient of the unknown function. The argument is via the obstacle method introduced by Caffarelli, Souganidis and Wang~\cite{CSW}.

We consider fully nonlinear equations of the form
\begin{equation}\label{e.pde}
F\!\left( D^2u^\ep,Du^\ep, \frac x\ep, \omega \right) = 0 \quad \mbox{in} \ U \subseteq \Rd,
\end{equation}
where $F$ is a uniformly elliptic, Lipschitz continuous, stationary-ergodic operator $F$ (the precise assumptions are given below). The homogenization result (Theorem~\ref{H} below) states that, almost surely, the solutions $u^\ep(x,\omega)$ of~\eqref{e.pde}, subject to an appropriate boundary condition, converge uniformly as $\ep \to 0$ to the (deterministic) solution $u$ of
\begin{equation*}\label{}
\overline F(D^2u,Du) = 0 \quad \mbox{in} \ U, 
\end{equation*}
for a uniformly elliptic operator $\overline F$.

A result like this was first proved in the fully nonlinear setting by Caffarelli, Souganidis and Wang~\cite{CSW}, who introduced a new method for obtaining stochastic homogenization of nonlinear equations based on an obstacle problem. They observed that, while the ``free" solutions of fully nonlinear equations do not possess an obvious linear or subadditive structure, which is needed to apply the ergodic theorem and thus to homogenize, the corresponding obstacle problem solutions do. Using clever arguments based on the regularity theory for such equations, they were then able to control the ``free" solutions with those of the obstacle problem sufficiently well to obtain almost sure homogenization in the case that $F$ does not depend on the gradient $Du^\ep$.

In the general case that $F$ may depend on $Du^\ep$, the arguments of~\cite{CSW} only imply that the ``approximate correctors" (the solutions of~\eqref{e.two} below) homogenize \emph{in probability}, and it has been an open problem to obtain the fully homogenization result in the almost sure sense. The trouble is that uniform bounds on the gradients of the ``approximate correctors," which are necessary for a straightforward application of the perturbed test function argument, are not easy to obtain: see the discussion on page 347 of~\cite{CSW}.

In the present paper, we resolve the difficulty with the gradient dependence and give the first complete proof of almost sure homogenization for general equations of the form~\eqref{e.pde}. The idea is to obtain the desired gradient bounds for a new approximate corrector, constructed by approximating the obstacle problem solutions by their infimal convolutions and then using the fact that the relevant contact sets spread evenly. Even in the gradient-independent setting, our approach permits us to give a considerably simplified presentation of the results in~\cite{CSW}.  

We proceed with the precise statement of the homogenization result.

\subsection*{The assumptions}

We consider Euclidean space $\Rd$ in dimension $\d\geq 1$. The \emph{random environment} consists of a given probability space $(\O,\F,\P)$ and a measure-preserving ergodic action $\tau = \left( \tau_y \right)_{y\in \Rd}$ of $\Rd$ on $\O$. Precisely, $\tau_y:\O\to\O$ is an $\F$-measurable map such that $\P\circ \tau_y = \P$ and $\tau_y\circ\tau_z = \tau_{y+z}$ for all $y,z\in \Rd$ and
\begin{equation}\label{erg}
\tau_y A = A \ \ \mbox{for every} \ y\in \Rd \qquad \mbox{implies that} \qquad \P[A] = 0 \ \ \mbox{or} \ \ \P[A] =1.
\end{equation}

We require that the fully nonlinear operator $F:\Sy\times \Rd \times \Rd \times \O \to \R$ satisfies each of the following three conditions:

\begin{enumerate}

\item[(F1)] {\it Stationarity and ergodicity}: for all $(M,p,\omega)\in \Sy\times \Rd \times \O$ and $y,z\in \Rd$,
\begin{equation*}\label{}
F(M,p,y,\tau_z\omega) = F(M,p,y+z,\omega),
\end{equation*}
where $\tau=(\tau_y)_{y\in\Rd}$ is as above and in particular  satisfies~\eqref{erg}.

\medskip

\item[(F2)] {\it Uniform ellipticity and Lipschitz continuity}: there exist constants $\gamma > 0$ and $0 < \lambda \leq \Lambda$ such that, for all $(M,p,z), (N,q,w) \in \Sy\times \Rd\times \R$ and $(y,\omega)\in\Rd\times\O$,
\begin{multline*}
\qquad \qquad \Pu^{-}_{\lambda,\Lambda}(M-N) - \gamma |p-q| \leq F(M,p,y,\omega) - F(N,q,y,\omega)\\ \leq \Pu^{+}_{\lambda,\Lambda} (M-N) + \gamma |p-q|.
\end{multline*}
\end{enumerate}

\noindent Here $\Pu^{\pm}_{\lambda,\Lambda}$ are the usual Pucci extremal operators, defined for each $M\in \Sy$ by
\begin{equation*}\label{}
\Pu^+_{\lambda,\Lambda}(M) := \Lambda \tr(M_-) - \lambda \tr(M_+) \quad \mbox{and} \quad\Pu^-_{\lambda,\Lambda}(M) : = \lambda \tr(M_-) - \Lambda \tr(M_+),
\end{equation*}
where $M_\pm$ are such that $M_\pm \geq 0$, $M= M_+-M_-$ and $M_-M_+ = 0$.

\begin{enumerate}
\item[(F3)] {\it Regularity and boundedness in the microscopic variable}: for every $R> 0$,
\begin{multline*}\label{}
\qquad \qquad \left\{ F(M,p,\cdot,\omega) \,:\, (M,p,\omega)\in\Sy\times\Rd\times\O, \ |M|, |p| \leq R \right\} \\ \mbox{is uniformly bounded and equicontinuous on} \ \ \Rd.
\end{multline*}
Moreover, there exists a modulus $\rho:[0,\infty) \to [0,\infty)$ and a constant $\sigma > \frac12$ such that, for all $(M,p,\omega)\in \Sy\times \Rd\times \Omega$ and $y,z\in \Rd$,
\begin{equation*}\label{}
\left|F(M,p,y,\omega) - F(M,p,z,\omega) \right| \leq \rho\big( (1+|M|)|y-z|^\sigma \big).
\end{equation*} 
\end{enumerate}
The reason for the last statement of (F3) is that, in light of (F1), it implies that the comparison principle holds for each of the operators $F(\cdot,\cdot,\cdot,\omega)$ with $\omega\in \Omega$ (see~\cite{CIL}).

\subsection*{The main result}

We state the homogenization result for the Dirichlet problem
\begin{equation}\label{bvp}
\left\{ \begin{aligned} 
&  F\!\left(D^2u^\ep,Du^\ep,\frac x\ep,\omega\right) = 0 & \mbox{in} & \ U, \\
& u^\ep = g & \mbox{on} & \ \partial U.
\end{aligned} \right.
\end{equation}
Here $U \subseteq \Rd$ is a bounded Lipschitz domain and $g\in C(\partial U)$, and the equation is understood in the viscosity sense (see~\cite{CIL,CC}).

By straightforward modifications of our argument, we may homogenize essentially any other well-posed problem involving the operator $F$, including parabolic equations subject to appropriate boundary and/or initial conditions. The arguments also extend easily to equations with more general dependence, such as
\begin{equation*}\label{}
F\!\left(D^2u^\ep,Du^\ep, u^\ep, x, \frac x\ep,\omega \right) = 0,
\end{equation*}
as well as, for example, equations with quadratic dependence in the gradient. Since these extensions present no additional difficulties, we focus on~\eqref{bvp} to avoid burdensome notation. 

\begin{thm}\label{H}
Assume (F1), (F2) and (F3). Then there exists an event $\Omega_0\in \F$ of full probability and a function $\overline F:\Sy \times \Rd  \to \R$ which satisfies 
\begin{equation*}\label{}
\Pu^-_{\lambda,\Lambda}(M-N) - \gamma |p-q|  \leq \overline F(M,p) - \overline F(N,q) \leq \Pu^+_{\lambda,\Lambda} (M-N) + \gamma|p-q|
\end{equation*}
such that, for every $\omega\in \Omega_0$, every bounded Lipschitz domain $U\subseteq \Rd$ and each $g\in C(\partial U)$, the unique solution $u^\ep$ of the boundary value problem~\eqref{bvp} satisfies
\begin{equation*}\label{}
\lim_{\ep \to 0} \sup_{x\in U} \left| u^\ep(x,\omega) - u (x)\right| = 0,
\end{equation*}
where $u\in C(\overline U)$ is the unique solution of the Dirichlet problem
\begin{equation}\label{Ee}
\left\{ \begin{aligned} 
& \overline F(D^2u,Du) = 0 & \mbox{in} & \ U, \\
& u = g & \mbox{on} & \ \partial U.
\end{aligned} \right.
\end{equation}
\end{thm}

\subsection*{Literature review}
The homogenization of elliptic equations in random media originated in the work of Papanicolaou and Varadhan~\cite{PV1,PV2} and Kozlov~\cite{K1,K2} about three decades ago. Linear equations are somewhat simpler to analyze since they possess a dual structure. Indeed, the method of~\cite{PV1,PV2} relies heavily on the existence of invariant measures, which are unavailable in the nonlinear setting. The obstacle method of~\cite{CSW} has since been used by Caffarelli and Souganidis~\cite{CS} to obtain a quantitative homogenization result for fully nonlinear equations under a mixing hypothesis, including a logarithmic rate of convergence, by Schwab~\cite{Sc2} in the setting of (nonlinear) nonlocal equations, and by the authors~\cite{AS} for fully nonlinear equations which are not uniformly elliptic.

\subsection*{Outline of the paper}
In the next section we briefly sketch the main ideas, introduced in~\cite{CSW}, underlying the obstacle problem approach to homogenization. In Section~\ref{PO} we give a succinct construction of the effective equation and demonstrate several of its inherited properties, including uniform ellipticity. In the last section we present the proof of Theorem~\ref{H} based on the perturbed test function method.

\section{A brief overview of the main ideas}

To summarize the concepts underlying the homogenization argument, we drop dependence on the gradient and consider the problem 
\begin{equation}\label{e.one}
\left\{ \begin{aligned}
& F\left(D^2u^\ep,\frac x\ep,\omega\right) = \alpha & \mbox{in} & \ B_1, \\
& u^\ep = 0 & \mbox{on} & \ \partial B_1.
\end{aligned} \right.
\end{equation}
If we rescale so that the microscopic scale is of unit order, we obtain the problem 
\begin{equation}\label{e.two}
\left\{ \begin{aligned}
& F\left(D^2u_r,y,\omega\right) = \alpha & \mbox{in} & \ B_r, \\
& u_r = 0 & \mbox{on} & \ \partial B_r,
\end{aligned} \right.
\end{equation}
for $r> 0$ very large. If~\eqref{e.one} homogenizes, then in terms of~\eqref{e.two} this means that $u_r(0) \approx r^2 f(\alpha)$ for large $r$, where $f$ is a strictly increasing function of $\alpha$. Assuming we could prove that $r^{-2} u_r(0) \rightarrow f(\alpha)$ as $r\to \infty$, we could then identify $\overline F(0)$ as the (necessarily unique) value of $\alpha$ for which $f(\alpha) = 0$. With this choice of $\alpha$, $u_r$ is ``flat" in the sense that $r^{-2} u_r(0) \approx 0$, and it turns out that this is precisely what we need to prove homogenization by the perturbed test function method. In short, it says that, for large $r$, $u_r$ is a  ``good approximate corrector." 

The main difficulty is precisely to show that $r^{-2} u_r(0)$ has a limit as $r\to \infty$, since the problem~\eqref{e.two} does not possess a structure amenable to the ergodic theorem. The idea of~\cite{CSW} is to instead consider the \emph{obstacle problem}
\begin{equation}\label{e.three}
\left\{ \begin{aligned}
& \min \left\{ F\left(D^2v_r,y,\omega\right)- \alpha, v_r\right\} = 0 & \mbox{in} & \ B_r, \\
& v_r = 0 & \mbox{on} & \ \partial B_r.
\end{aligned} \right.
\end{equation}
Clearly the solution of~\eqref{e.three} satisfies $v_r \geq 0$, since the obstacle (the zero function) prohibits it from being negative. Where it is positive, $v_r$ is unconstrained and so solves the same equation as the one for $u_r$. We therefore think of $v_r$ as being similar to $u_r$, but with some additional ``help" staying nonnegative. The amount of ``help" can be measured in terms of the Lebesgue measure of the \emph{contact set} $\{ v_r =0 \}$, and a crucial observation of~\cite{CSW} is that, due to the comparison principle, \emph{this quantity is subadditive}. Therefore, the ergodic theorem applies and we can conclude that the contact set takes up a deterministic proportion of $B_r$ as $r\to \infty$. 

To identify $\overline F(0)$, we start from $\alpha=-\infty$ and increase $\alpha$ until $v_r$ ``doesn't need help" staying nonnegative, that is, until the limiting proportion of the contact set vanishes for the first time. Using the regularity theory for uniformly elliptic equations and comparing $v_r$ to $u_r$ with the ABP inequality, it can then be shown that, for precisely this value of $\alpha$, $\lim_{r\to \infty} r^{-2} u_r(0) = 0$, as desired.

The extra difficulty that occurs if $F$ depends on the gradient is that in this case the perturbed test function method also requires that $r^{-1} Du_r(0) \rightarrow 0$ as $r\to \infty$. Obtaining the analogue of this condition is easy for periodic homogenization, but in the random setting the standard elliptic estimates do not yield it. To resolve this issue, we introduce infimal convolution approximations of $v_r$ and use them as ``approximate one-sided correctors" in the perturbed test function argument to gain extra control over the gradient. We take advantage of the fact that contact set ``spreads evenly" on large scales (see Lemma~\ref{l.spread}) to show that these approximations satisfy precisely the required gradient bound. Unlike~\cite{CSW}, we make no use of the ``free" problem~\eqref{e.two} in our proof of homogenization.

\section{The obstacle problem and the identification of $\overline F$} \label{PO}

In this section, following the ideas of~\cite{CSW}, we construct the effective operator $\overline F$ by applying the subadditive ergodic theorem to a quantity involving the obstacle problem.

\subsection*{The obstacle problem}

We begin with a discussion of the basic properties of the obstacle problem. Succinct proofs of the following standard facts can be found for example in~\cite{CSW} as well as the appendix of~\cite{AS}. The obstacle problem (with the zero function as the obstacle) is:
\begin{equation}\label{e.o}
 \min\left\{ F(D^2u,0,y,\omega) , u \right\} = 0.
\end{equation}
It is easy to see that~\eqref{e.o} satisfies a comparison principle. That is, if $V\in \L$ ($:=$ set of bounded Lipschitz domains of $\Rd$) and $u_1,-u_2\in C\left(\overline V\right)$ are such that
\begin{equation*}\label{}
 \min\left\{ F(D^2u_1,0,y,\omega) , u_1 \right\} \leq 0 \leq  \min\left\{ F(D^2u_2,0,y,\omega) , u_2 \right\} \quad \mbox{in} \ V,
\end{equation*}
then $u_1\leq u_2$ on $\partial V$ implies that $u_1 \leq u_2$ in $V$. The Perron method (with the help of some standard boundary barriers) then yields, for each $V\in \L$, a unique viscosity solution $w(\cdot,\omega\,;V,F) \in C(\overline V)$ of the boundary value problem
\begin{equation}\label{obst}
\left\{ \begin{aligned} 
& \min\left\{ F(D^2w,0,y,\omega) , w \right\} = 0 & \mbox{in} & \ V, \\
& w = 0 & \mbox{on} & \ \partial V.
\end{aligned} \right.
\end{equation}
The function $w$ can be identified either as the minimal nonnegative supersolution of $F(D^2u,0,y,\omega) \geq 0$, or alternatively as the maximal subsolution of $F(D^2u,0,y,\omega) \leq k \chi_{\{ u \leq 0 \}}$ that is nonpositive on $\partial V$, where $k:= \sup_{y\in V} F(0,0,y,\omega)$ and $\chi_E$ denotes the characteristic function of a set $E \subseteq \Rd$. In particular, with $k$ as above, $w(\cdot,\omega\,;V,F)$ satisfies
\begin{equation}\label{e.obst}
0 \leq F(D^2w,0,y,\omega) \leq k \chi_{\{ w=0\}} \quad \mbox{in} \ V. 
\end{equation}
Finally, we remark that if $F_1$ and $F_2$ are two operators satisfying our assumptions, then, for every $V$,
\begin{equation}\label{e.mono1}
F_1 \leq F_2 \qquad \mbox{implies that} \qquad w(\cdot,\omega\,;V,F_1) \geq w(\cdot,\omega\,;V,F_2). 
\end{equation}
This is immediate from the comparison principle, or alternatively from the characterization of $w$ as the minimal supersolution. 
The obstacle problem possesses a second monotonicity property, which is also immediate from either the comparison principle or the minimal supersolution characterization, which states that 
\begin{equation}\label{e.mono2}
V \subseteq W \qquad \mbox{implies that} \qquad w(\cdot,\omega\,;V,F) \leq w(\cdot,\omega\,;W,F) \quad \mbox{in} \ V.
\end{equation}

In part due to the right side of~\eqref{e.obst}, the set of points at which~$w$ vanishes plays an key role in what follows, and so we denote it by
\begin{equation*}\label{}
\Cs(\omega\,;V,F):= \left\{ y\in V \, : \, w(y,\omega\,;V,F) = 0 \right\},
\end{equation*}
We call $\Cs(\omega\,;V,F)$ the \emph{contact set} since it is the set where $w$ touches the obstacle. Its Lebesgue measure is a very important quantity, due to the sublinear structure it possesses, and we write 
\begin{equation}\label{cntctset}
\m(V,\omega\,;F): = \left|\Cs(\omega\,;V,F) \right|.
\end{equation}
The contact set inherits two monotonicity properties from the obstacle problem: namely that
\begin{equation}\label{e.mono1b}
F_1 \leq F_2 \qquad \mbox{implies that} \qquad \Cs(\omega\,;V,F_1) \subseteq \Cs(\omega\,;V,F_2)
\end{equation}
and
\begin{equation}\label{Cs-mono}
V \subseteq W \qquad \mbox{implies that} \qquad \Cs(\omega\,;W,F) \cap V \subseteq \Cs(\omega\,; V,F),
\end{equation}
which follow immediately from~\eqref{e.mono1} and~\eqref{e.mono2}, respectively.

The following proposition asserts that, on large scales, the contact set occupies a limiting proportion of the underlying domain, and this proportion is (almost surely) deterministic and does not depend on the domain. This is obtained by an application of the multiparameter subadditive ergodic theorem, and it is the most important limit we take (as well as the only use of the ergodic theorem) in the course of the proof of Theorem~\ref{H}. The argument is essentially the same as that of~(3.3) in~\cite{CSW}.

\begin{prop} \label{p.obstH}
There exists an event $\Cl[O]{obstHO}(F) \in \F$ of full probability and a deterministic constant $\overline\m (F)\in \R$ such that, for every $\omega\in \Cr{obstHO}(F)$ and $V \in \L$,
\begin{equation}\label{obstHe}
\lim_{t\to\infty} \frac{\m(tV,\omega\,;F)}{|tV|} = \overline\m(F).
\end{equation}
\end{prop}
\begin{proof}
We check that $\m$ satisfies the hypotheses of the multiparameter subadditive ergodic theorem (the version we refer to can be found in~Dal Maso and Modica~\cite{DM2}, see also the remarks following Proposition 2.2 in~\cite{AS}). 

Immediate from~\eqref{Cs-mono} is the subadditivity of~$\m$. That is, for all $V, V_1,\ldots, V_{k} \in \L$ such that $V_1,\ldots,V_k$ are pairwise disjoint, $\cup_{j=1}^k V_j \subseteq V$ and $| V \setminus \cup_{j=1}^k V_j | = 0$, we have
\begin{equation}\label{msa}
\m(V,\omega\,;F) \leq \sum_{j=1}^k m(V_j,\omega\,;F). 
\end{equation}
According to (F1), $\m$ is stationary. That is, for every $y\in \Rd$ and $V\in \L$,
\begin{equation*}\label{}
\m(V,\tau_y\omega\,;F) = \m(y+V,\omega\,;F).
\end{equation*}
We may easily extend the definition of $\m$ to the class $\BS$ of bounded Borel subsets of $\Rd$ by defining, for every $A \in \BS$, 
\begin{equation*}\label{}
\tilde \m(A,\omega\,;F):= \inf\left\{ \m(V,\omega\,;F)\,:\, V \in \L \ \mbox{and} \ A \subseteq V \right\}.
\end{equation*}
This extension agrees with $\m$ on $\L$ by~\eqref{Cs-mono} and the subadditivity and stationarity properties are preserved.

An application of~\cite[Proposition 1]{DM2} now yields the proposition.
\end{proof}

The limit~\eqref{obstHe} suffices to define the effective operator $\overline F$, but not to prove homogenization. We require something slightly more precise, namely that not only does the contact set, on large scales, occupy a limiting proportion of its domain, but it also spreads around evenly in the domain. The precise statement is the following lemma, which is obtained from~\eqref{Cs-mono} and~\eqref{obstHe}. The proof is essentially the same as~\cite[Lemma 3.4]{CSW}.

\begin{lem} \label{l.spread}
For every~$\omega\in\Cr{obstHO}(F)$ and $V,W \in \L$ with $\overline W \subseteq V$, 
\begin{equation}\label{e.limsubs}
\lim_{t\to \infty} \frac{ |\Cs(\omega\,;tV,F) \cap tW|}{|tW|} = \overline m(F).
\end{equation}
\end{lem}
\begin{proof}
Let $U:= V\setminus W\in \L$ and fix $\omega\in\Cr{obstHO}(F)$. Observe that~\eqref{Cs-mono} gives
\begin{equation}\label{e.spup}
\limsup_{t\to \infty} \frac{|\Cs(\omega\,;tV,F) \cap tW|}{|tW|} \leq\lim_{t\to \infty} \frac{|\Cs(\omega\,;tW,F)|}{|tW|} = \overline m(F).
\end{equation}
In the same way, we have
\begin{equation*}\label{}
\limsup_{t\to \infty} \frac{|\Cs(\omega\,;tV,F) \cap tU|}{|tU|} \leq \overline m(F).
\end{equation*}
Hence
\begin{multline*}\label{e.spdn}
\liminf_{t\to \infty} \frac{|\Cs(\omega\,;tV,F) \cap tW|}{|tW|} =\liminf_{t\to \infty} \frac{|\Cs(\omega\,;tV,F) \cap tV|-|\Cs(\omega\,;tV,F) \cap tU|}{|tW|} \\ \geq \left( \frac{|V|}{|W|} - \frac{|U|}{|W|} \right) \overline m(F) = \overline m(F).
\end{multline*}
Combined with~\eqref{e.spup}, this implies~\eqref{e.limsubs}.
\end{proof}

\subsection*{The effective nonlinearity}
We now define the effective operator $\overline F$ and discuss some of its elementary properties. It is prescribed in terms of the limiting proportions $\overline m\left(F_{M,p}-\alpha\right)$ given in Proposition~\ref{p.obstH}, where $\alpha\in \R$ and the operator $F_{M,p}:\Sy\times\Rd\times\Rd\times \O$ is define for each fixed $(M,p)\in \Sy\times\Rd$ by
\begin{equation}\label{}
F_{M,p}(N,q,y,\omega):= F_{M,p}(M+N,p+q,y,\omega).
\end{equation}
Note that each operator $F_{M,p}$ satisfies assumptions~(F1),~(F2) and~(F3) and so in particular Proposition~\ref{p.obstH} applies.
\begin{definition}

We define the effective nonlinearity $\overline F:\Sy\times \Rd\to\R$ by
\begin{equation}\label{Fbar}
\overline F(M,p):= \sup\left\{ \alpha \in \R \, : \, \overline \m\left(F_{M,p} - \alpha \right) > 0 \right\}.
\end{equation}
\end{definition}

To check that $\overline F$ is well-defined and finite, we first observe that, by the characterization of the obstacle problem solution $w$ as the minimal supersolution, we immediately obtain that
\begin{equation*}\label{}
\inf_{y\in V} F(0,0,y,\omega) \geq 0 \qquad \mbox{implies that} \qquad \Cs(\omega\,;V,F) = V
\end{equation*}
and
\begin{equation*}\label{}
\sup_{y\in V} F(0,0,y,\omega)  < 0  \qquad \mbox{implies that} \qquad \Cs(\omega\,;V,F) = \emptyset.
\end{equation*}
It follows from these that
\begin{equation}\label{easybnds}
\essinf_{\omega\in \Omega} F(M,p,0,\omega) \leq \overline F(M,p) \leq \esssup_{\omega\in \Omega} F(M,p,0,\omega).
\end{equation}
The first monotonicity property~\eqref{e.mono1b} of the obstacle problem implies that the map $\alpha \mapsto \overline m(F-\alpha)$ is a nonincreasing function and therefore 
\begin{equation}\label{e.pspp}
\alpha < \overline F(M,p) \quad \mbox{implies that} \quad\overline m\left(F_{M,p}-\alpha\right) > 0
\end{equation}
and
\begin{equation}\label{e.zpp}
\alpha > \overline F(M,p) \quad \mbox{implies that} \quad\overline m\left(F_{M,p}-\alpha\right) = 0.
\end{equation}
Also from~\eqref{e.mono1b} we see that, if $F_1$ and $F_2$ each satisfy~(F1),~(F2) and~(F3), then, for each $p\in \Rd$,
\begin{multline}\label{obvmono}
\sup_{(M,\omega)\in \Sy\times\O}  \left( F_1(M,p,0,\omega) - F_2(M,p,0,\omega) \right)\leq 0 \\ \mbox{implies that}\quad \overline F_1(\cdot,p) \leq \overline F_2(\cdot,p).
\end{multline}
It is also clear that adding constants commutes with the operation $F \mapsto \overline F$. From these facts a number of properties of $\overline F$ are immediate, the ones inherited from uniform properties of $F$. We summarized a few of these in the following lemma. 

\begin{lem}\label{Fbarprop}
For each $(M,p),(N,q)\in \Sy\times\Rd$,
\begin{equation}\label{Fdegell}
\Pu^-_{\lambda,\Lambda}(M-N) -\gamma|p-q| \leq \overline F(M,p) - \overline F(N,q) \leq \Pu^+_{\lambda,\Lambda}(M-N) + \gamma|p-q|.
\end{equation}
In particular, $\overline F$ is Lipschitz on $\Sy\times\Rd$. Moreover, if $F$ is positively homogeneous of order one, odd, or linear in one or both of the variables $(M,p)$, then $\overline F$ possesses the same property. 
\end{lem}

\begin{proof}
Each of the properties are proved using the comments before the statement of the proposition. To prove~\eqref{Fdegell}, we simply observe that, according to (F2), for all $(Y,y,\omega)\in \Sy\times\Rd\times\Omega$, $M,N\in \Sy$ with $M\leq N$, and $p,q\in \Rd$,
\begin{multline*}\label{}
F(N+Y,q,y,\omega) + \lambda \tr(N-M) - \gamma|p-q| \leq F(M+Y,p,y,\omega) \\ \leq F(N+Y,q,y,\omega) + \Lambda \tr(N-M) + \gamma|p-q|
\end{multline*}
and then apply~\eqref{obvmono}. It is obvious that $\overline F$ inherits the properties of positive homogeneity and oddness from $F$, and linearity follows from these. 
\end{proof}

Another property of the operation $F \mapsto \overline F$, which is less obvious than those of Lemma~\ref{Fbarprop}, is that it commutes with odd reflection. The odd reflection operator~$\od{F}$ is defined by
\begin{equation*}\label{}
\od{F}(M,p,y,\omega):= -F(-M,-p,y,\omega),
\end{equation*}
and it is straightforward to check that $\od{F}$ satisfies each of (F1), (F2) and (F3) if and only if $F$ does. Moreover, it is easy to see that odd reflection simply exchanges sub- and supersolutions, that is,
\begin{equation}\label{e.oddrefl}
F(D^2u,Du,y,\omega) \leq 0 \quad \iff \quad v:=-u \ \mbox{satisfies} \ \ \od{F}(D^2v,Dv,y,\omega) \geq 0.
\end{equation}
In the next lemma, we show that $F\mapsto \od{F}$ commutes with $F\mapsto \overline F$, a fact we use in the proof of Theorem~\ref{H}.

\begin{lem} \label{l.odd}
$\od{\overline F} = \od{\left(\overline F\right)}$.
\end{lem}
\begin{proof}
Suppose on the contrary that, for some $(M,p)\in \Sy \times\Rd$ and $\alpha,\beta \in \R$,
\begin{equation*}\label{}
\od{\overline F}(M,p) < \alpha < \beta < \od{\left( \overline F\right) }\!(M,p).
\end{equation*}
That is, for $\alpha < \beta$, we have
\begin{equation*}\label{}
\od{\overline F}(M,p) < \alpha \quad \mbox{and} \quad \overline F(-M,-p) < -\beta.
\end{equation*}
According to~\eqref{e.zpp}, this implies that
\begin{equation}\label{e.cszero}
\overline m\left(\od{F}_{M,p} -\alpha \right) = \overline m\left( F_{-M,-p} + \beta\right) = 0.
\end{equation}
Fix $\omega\in \Cr{obstHO}\left(\od{F}_{M,p} -\alpha \right) \cap \Cr{obstHO}\left( F_{-M,-p} + \beta\right)$ and consider the function
\begin{multline}\label{e.ur}
u_r(y):=  w\!\left(y,\omega\,;B_r,F_{-M,-p} + \beta\right) + w\!\left(y,\omega\,;B_r,\od{F}_{M,p} -\alpha\right) \\ +  \frac{\beta-\alpha}{2\d\Lambda}\left( r^2 - |y|^2 \right).
\end{multline}
Denote the first two functions on the right side of~\eqref{e.ur} by $w_{1,r}(y)$ and $w_{2,r}(y)$, respectively. Since $w_{1,r},w_{2,r}\geq 0$, we clearly have
\begin{equation}\label{e.tocdct}
\liminf_{r\to \infty} r^{-2} u_r(0) \geq \frac{\beta-\alpha}{2\d\Lambda} > 0.
\end{equation}
Let $E_{1,r}:= \Cs(\omega\,;B_r,F_{-M,-p}+\beta)$ and $E_{2,r}:= \Cs(\omega\,;B_r,\od{F}_{M,p} -\alpha)$ denote the contact sets for $w_{1,r}$ and $w_{2,r}$, respectively. Formally, using~(F1),~\eqref{e.obst} and~\eqref{e.oddrefl}, we have
\begin{align*}\label{}
\Pu^-_{\lambda,\Lambda}(D^2u_r) & \leq \Pu^-_{\lambda,\Lambda}(D^2(w_{1,r} + w_{2,r})) + \Pu^+_{\lambda,\Lambda}\left( - \frac{\beta-\alpha}{\d\Lambda} I_\d \right)  \\
& \leq F(-M+D^2w_{1,r},-p,y,\omega) - F(-M-D^2w_{2,r},-p,y,\omega) + \left(\beta-\alpha \right)  \\
& \leq \left( -\beta + k \chi_{E_{1,r}}  \right) - \left( -\alpha - k\chi_{E_{2,r}}\right) + \left(\beta-\alpha \right) \\
& \leq 2k\chi_{E_{1,r} \cup E_{2,r}}
\end{align*}
in $B_r$, where $k:=\esssup_\Omega F(-M,-p,0,\cdot)$. This string of inequalities is rigorous (see for example the remarks in Section~2 of~\cite{AS} for a proof of the standard fact that inequalities are transitive in the viscosity sense). The ABP inequality (c.f.~\cite{CC}) applied to the function $\tilde u_r(x): = r^{-2} u_r(rx)$ then yields that 
\begin{equation*}\label{}
r^{-2} u_r(0) \leq C k r^{-1} \left( m(B_r,\omega\,;F_{-M,-p}+\beta) + m(B_r,\omega\,;\od{F}_{M,p} -\alpha) \right)^{1/\d}.
\end{equation*}
Sending $r\to \infty$ and using~\eqref{e.cszero}, we obtain that $\limsup_{r\to \infty} r^{-2} u_r(0) \leq 0$, which is in violation of~\eqref{e.tocdct}. 
\end{proof}

\section{The proof of homogenization}

In this section, we complete the proof of Theorem~\ref{H} using a modified perturbed test function argument based on the method introduced in the context of nonlinear homogenization by Evans~\cite{E}. 

In order to gain some control on the gradient of the approximate correctors, we modify the obstacle problem solution $w$ by introducing, for each $\delta > 0$, the infimal convolution approximation
\begin{equation}\label{e.wdel}
w^\delta(y,\omega\,;V,F):= \inf_{z\in V}\left\{ w(z,\omega\,;V,F) + \frac1{2\delta}|y-z|^2 \right\}.
\end{equation}
The function $w^\delta$ satisfies the differential inequality
\begin{equation}\label{e.wdeleq}
F(D^2w^\delta,Dw^\delta,y,\omega) \geq -c_\delta \quad \mbox{in} \ V_{s_\delta} : = \left\{ y\in V \, : \, \dist(y,\partial V) > s_\delta \right\},
\end{equation}
for $c_\delta, s_\delta >0$ such that $c_\delta, s_\delta \to 0$ as $\delta \to 0$. This is routine to check using the elementary properties of infimal convolution and (F3), and we refer to~\cite{CIL} for details. 

An important property of the functions $w^\delta(\cdot,\omega\,;V,F)$ is that they are locally semiconcave, and therefore locally Lipschitz and differentiable Lebesgue almost everywhere in $V$. In fact, they are differentiable at any point at which they can be touched from below by a smooth function. See~\cite{CIL} for details.

It is immediate from~\eqref{e.wdel} and the nonnegativity of $w$ that $0 \leq w^\delta \leq w$ and the infimal convolution leaves the contact set undisturbed, that is,
\begin{equation}\label{e.csunch}
\left\{ y \in V\,:\, w^\delta(y,\omega\,;V,F) = 0 \right\} = \left\{ y \in V\,:\, w(y,\omega\,;V,F) = 0 \right\} = \Cs(\omega\,;V,F).
\end{equation}
This implies in particular that $Dw^\delta(\cdot,\omega\,;V,F)$ exists and vanishes on $\Cs(\omega\,;V,F)$. We next present a generalization of this fact, stating that we can control $Dw^\delta$ in terms of the distance to the contact set. Since the contact set ``spreads," this will prove to be useful.

\begin{lem}
At any point $y\in V$ at which $w^\delta(\cdot,\omega\,;V,F)$ is differentiable,
\begin{equation}\label{e.gradcs}
\left|Dw^\delta(y,\omega\,;V,F)\right| \leq \frac1{\delta} \dist\left( y, \Cs(\omega\,;V,F) \right).
\end{equation}
\end{lem}
\begin{proof}
For simplicity, we suppress the dependence of our functions and sets on $(\omega,V,F)$. Since $w$ vanishes on the boundary of $V$, the infimum in~\eqref{e.wdel} is attained at some point $z \in \bar V$, and by comparing $z$ to the nearest point to $y$ at which $w$ vanishes, we deduce
\begin{equation*}\label{}
w^\delta(y) = w(z) + \frac1{2\delta}|z-y|^2 \leq \frac1{2\delta}\left( \dist\left( y, \Cs \right) \right)^2.
\end{equation*}
In particular, since $w\geq 0$,
\begin{equation*}\label{}
|z-y| \leq \dist\left( y, \Cs\right).
\end{equation*}
If $y'\in B_r(y)$, then we have
\begin{equation*}\label{}
w^\delta(y') \leq w(z) + \frac1{2\delta}|z-y'|^2 \leq w^\delta(y) + \frac{1}{2\delta}\left( |y-y'|^2 + 2 |z-y||y-y'| \right)
\end{equation*}
and thus
\begin{equation*}\label{}
\sup_{B_r(y)} \left( w^\delta - w^\delta(y) \right) \leq \frac{1}{\delta}\left( \frac 12 r^2 + \dist(y,\Cs) r \right)
\end{equation*}
Dividing by $r$ and sending $r\to 0$ yields the lemma.
\end{proof}

The standard H\"older estimates and~\eqref{e.gradcs}, combined with Lemma~\ref{l.spread}, yield the following result. It is~\eqref{e.wdelsq} which asserts that the $w^\delta$'s are ``flat enough" for use in the perturbed test function method, and~\eqref{e.gcontrol} which permits us to handle gradient dependent equations in the proof of Lemma~\ref{l.pertstp} below. Essentially, the lemma states that the functions $w^\delta$ are ``good enough approximate correctors." 

Before giving the lemma, we reveal the identity of the event $\Omega_0$ in the statement of Theorem~\ref{H}. We define $\Omega_0$ to be the intersection, over all $M\in \Sy$, $p\in \Rd$ and $a\in \R$, with rational entries, of the events $\Cr{obstHO}(F_{M,p}-a)$ and $\Cr{obstHO}(\od{F}_{M,p}-a)$. It is clear that $\P[\Omega_0] = 1$ since $\Omega_0$ is the countable intersection of events of full probability.

\begin{lem}\label{l.gcontrol}
Suppose that $(M,p)\in \Sy\times \Rd$ and $a\in \R$ are such that $\overline F(M,p)>a$. Then, for each $V\in \L$, $\omega\in\Omega_0$ and $\delta > 0$,
\begin{equation}\label{e.wdelsq}
\limsup_{\ep\to 0} \, \ep^2 \sup_{y\in \frac1\ep V} \left|w^\delta\!\left(y,\omega\,;\frac 1\ep V,F_{M,p}-a\right)\right| = 0
\end{equation}
and
\begin{equation}\label{e.gcontrol}
\limsup_{\ep \to 0} \sup_{y\in \frac1\ep V} \, \ep \left| Dw^\delta\left(y,\omega\,;\frac1\ep V,F_{M,p}-a\right) \right| = 0.
\end{equation}
\end{lem}
\begin{proof}
We first prove the lemma for $(M,p,a)$ with rational entries, and in this case we may assume with no loss of generality that $M=0$, $p=0$ and $a=0$.

Let $\eta > 0$ and select $x_1,\ldots,x_N \in V$ such that $\overline V$ is covered by the collection of balls $\{ B_{\eta}(x_j)\}_{j=1}^N$. According to~Lemma~\ref{l.spread},~\eqref{e.pspp} and the assumption that $a<\overline F(M,p)$, there exists $\ep(\eta) > 0$ such that, for every $0 < \ep <\ep(\eta)$,
\begin{equation}\label{e.cprus}
\Cs\left( \omega\,;\frac1\ep V,F \right) \cap \frac1\ep B_\eta(x_j) \neq \emptyset \quad \mbox{for every} \ j\in \{1,\ldots,N\}.
\end{equation}
Since $0\leq w^\delta \leq w$ and $w$ vanishes on the contact set, the standard $C^\alpha$ estimates (c.f.~\cite{CC}), properly scaled  and applied to $w$, using~\eqref{e.obst}, yield that, for some $\alpha >0$ and every $0< \ep < \ep(\eta)$,
\begin{equation}\label{e.smak1}
 \ep^2\sup_{y\in  V} \ \left| w^\delta\!\left(y,\omega\,;\frac 1\ep V,F\right) \right| \leq  \ep^2\sup_{y\in  V} \  w\! \left(y,\omega\,;\frac 1\ep V,F\right)  \leq C \eta^\alpha
\end{equation}
Letting $\ep \to 0$ and then $\eta\to 0$ in~\eqref{e.smak1} yields~\eqref{e.wdelsq}.
We also deduce from~\eqref{e.cprus} that, for every $0 < \ep<\ep(\eta)$ and $y\in \frac 1\ep V$,
\begin{equation}\label{e.ctsted}
\dist\left( y, \Cs\left(\omega\,; \frac1\ep V,F\right) \right) \leq \frac{2\eta}{\ep}.
\end{equation}
From~\eqref{e.gradcs} we deduce that, for every $0 < \ep<\ep(\eta)$,
\begin{equation*}
\sup_{y\in \frac 1\ep V} \, \ep \left|Dw^\delta\left(y,\omega\,;\frac1\ep V,F\right)\right| \leq \frac{2\eta}{\delta}.
\end{equation*}
We send $\ep \to 0$ and then $\eta\to 0$ to obtain~\eqref{e.gcontrol}. This completes the argument in the case that $(M,p,a)$ has rational entries. By the continuity of $\overline F$ given in Lemma~\ref{Fbarprop} and using \eqref{e.mono1} and~\eqref{e.mono1b}, we still have both~\eqref{e.smak1} and~\eqref{e.ctsted} with $F$ replaced by $F_{M,p}-a$ with arbitrary $(M,p,a) \in \Sy\times\Rd\times\R$. We may then conclude by arguing as above.
\end{proof}

The main step in the perturbed test function argument is encapsulated by the following lemma (the reader is encouraged to skip it and first read the proof of Theorem~\ref{H}). We remark that if $F$ does not depend on $p$, then the argument can be simplified further, since in this case we have no use for~\eqref{e.gcontrol} and we may use $w$ instead of $w^\delta$.

\begin{lem} \label{l.pertstp}
Fix $\omega\in \Omega_0$, $x_0\in \Rd$, $r_0> 0$ and $\phi\in C^\infty(B_{r_0}(x_0))$ and set $M:=D^2\phi(x_0)$ and $p:=D\phi(x_0)$. Also fix $a < \overline F(M,p)$ and define, for each $\delta, \ep > 0$,
\begin{equation}\label{e.pert}
\phi^{\ep,\delta}(x):= \phi(x) + \ep^2 w^\delta\left( \frac x\ep, \omega\,;\frac1\ep B_{r_0}(x_0),F_{M,p} -a \right).
\end{equation}
Let $\eta > 0$. Then there exists $0 <  r  < r_0$ and $\delta_0>0$ so that for each $0 < \delta < \delta_0$ there exists $\ep_0(\delta)> 0$ such that, for each $0 < \ep < \ep_0(\delta)$, the perturbed test function~$\phi^{\ep,\delta}$ satisfies the inequality
\begin{equation*}\label{e.pertstp}
F\left(D^2\phi^{\ep,\delta}, D\phi^{\ep,\delta},\frac x\ep,\omega \right) \geq a - \eta  \quad \mbox{in} \ B_{r}(x_0).
\end{equation*}
\end{lem}
\begin{proof}
Fix $\eta > 0$ and select $\psi \in C^\infty(B_{r_0}(x_0))$ and a point $x_1\in B_r(x_0)$, with $0<r<r_0$ to be determined below, such that
\begin{equation*}\label{}
x\mapsto \left(\phi^{\ep,\delta} - \psi\right) (x) \quad \mbox{has a strict local minimum at} \ x=x_1.
\end{equation*}
Expressing this in terms of $w^\delta$, we find that
\begin{multline}\label{e.wdtouch}
y\mapsto w^\delta\left(y,\omega\,;\frac1\ep B_{r_0}(x_0) , F_{M,p} -a  \right) - \frac1{\ep^2} \left( \psi( \ep y) - \phi(\ep y) \right) \\ \mbox{has a local minimum at} \ y = y_1:= \frac{x_1}{\ep}. 
\end{multline}
We fix $\delta_0>0$ small enough that, for each $0<\delta \leq \delta_0$, the constants~\eqref{e.wdeleq} satisfy $s_\delta \leq r_0 -r$ and $c_\delta \leq \frac12\eta$. Then for such $\delta$ we have
\begin{equation}\label{e.pmrplug}
F \! \left( M+D^2\psi(x_1) - D^2\phi(x_1) , p , \frac{x_1}{\ep},\omega  \right) \geq a - \frac12\eta. 
\end{equation}
Since $\phi$ is smooth, for small $r>0$ we have
\begin{equation*}\label{}
|D^2\phi(x_0)- D^2\phi(x_1)| \leq r  \left( \sup_{B_{r_0}(x_0)} |D^3\phi| \right),
\end{equation*}
which can be made as small as desired by shrinking $r$, and a similar bound holds for $|D\phi(x_0) - D\phi(x_1)|$. Observe that~\eqref{e.wdtouch} implies that $w^\delta\left(\cdot,\omega\; \frac1\ep B_{r_0}(x_0), F_{M,p} -a \right)$ is differentiable at~$y_1$ and
\begin{equation}\label{e.gradplug}
|D\phi(x_1) - D\psi(x_1)| = \ep \left |Dw^\delta\left(y_1,\omega\; \frac1\ep B_{r_0}(x_0), F_{M,p} -a \right) \right|.
\end{equation}
The quantity on the right of~\eqref{e.gradplug} is bounded from above by a quantity which tends to zero as $\ep\to 0$ (at a rate which depends on $\delta$) by Lemma~\ref{l.gcontrol}, which is applicable by the assumption that $a < \overline F(M,p)$. Therefore, these considerations and \eqref{e.pmrplug} together with the uniform continuity assumption in (F3) imply that if~$r>0$ and~$\delta>0$ are small enough then, for all sufficiently small $\ep>0$ (depending on $\delta$), we have
\begin{equation*}\label{}
F \! \left( D^2\psi(x_1) , D\psi(x_1) , \frac{x_1}{\ep},\omega  \right) \geq a  - \eta.
\end{equation*}
This completes the proof.
\end{proof}

We now complete the proof of the main result.

\begin{proof}[{\bf Proof of Theorem~\ref{H}}]
We fix  $\omega\in \Omega_0$, a bounded Lipschitz domain $U\in \L$ and $g \in C(\partial U)$. We first argue that, for every $x\in U$,
\begin{equation}\label{Hup}
\tilde u(x):= \limsup_{\ep \to 0} u^\ep(x,\omega_0) \leq u(x). 
\end{equation}
By the comparison principle, to prove~\eqref{Hup} it suffices to check that $\tilde u$ satisfies  
\begin{equation}\label{tiluss}
\left\{ \begin{aligned}
& \bar F(D^2\tilde u,D\tilde u) \leq 0 & \mbox{in} & \ U,\\
& \tilde u \leq g & \mbox{on} & \ \partial U.
\end{aligned} \right.
\end{equation}
That $\tilde u  = g$ on $\partial U$ is obtained by a routine barrier argument. To verify the PDE in~\eqref{tiluss}, we select a smooth test function $\phi\in C^2(U)$ and a point $x_0\in U$ such that
\begin{equation}\label{e.maxpt}
x \mapsto \left( \tilde u - \phi \right)(x) \quad \mbox{has a strict local maximum at} \quad x = x_0.
\end{equation}
We must show that $\overline F(D^2\phi(x_0), D\phi(x_0)) \leq 0$, and so arguing on the contrary, we set $M:=D^2\phi(x_0)$ and $p:=D\phi(x_0)$ and suppose that $\theta:= \overline F(M,p )  > 0$.

Since the local maximum of $\tilde u-\phi$ at $x_0$ is strict, there exists $r_0>0$ such that $B_{r_0}(x_0) \subseteq U$ and, for every $0 < r \leq r_0$,
\begin{equation}\label{stroom}
\left(\tilde u - \phi\right)(x_0) > \sup_{\partial B_r(x_0)} \left(\tilde u - \phi\right).
\end{equation}
Fix $\delta > 0$ to be selected below and let $\phi^{\ep,\delta}$ be as in~\eqref{e.pert} with $a:=\frac 12\theta$. By the definition of $\tilde u$ and~\eqref{e.wdelsq}, for each $0 < r \leq r_0$, there exists $\ep_r > 0$ such that, for every $0 < \ep < \ep_r$ and $r\leq s\leq r_0$,
\begin{equation}\label{stroom2}
\left(u^\ep - \phi^{\ep,\delta}\right)(x_0) > \sup_{\partial B_s(x_0)} \left(u^\ep - \phi^{\ep,\delta}\right).
\end{equation}
However, according to Lemma~\ref{l.pertstp}, for small enough $\delta, r,\ep > 0$ the function~$\phi^{\ep,\delta}$ satisfies the inequality
\begin{equation}\label{e.phiep}
F\left(D^2\phi^{\ep,\delta}, D\phi^{\ep,\delta}, \frac x\ep, \omega_0 \right) \geq \frac 14 \theta \quad \mbox{in} \ B_{r}(x_0).
\end{equation}
In light of the equation satisfied by $u^\ep$, this gives the desired contradiction, since it renders~\eqref{stroom2} in violation of the comparison principle. 

To prove that $\liminf_{\ep \to 0} u^\ep(x,\omega_0) \geq u(x)$, we simply replace $u^\ep$ and $u$ by $-u^\ep$ and $-u$, apply Lemma~\ref{l.odd}, and argue as above. 

We have shown that $\lim_{\ep \to 0} u^\ep(x,\omega) = u(x)$ for all $x\in \bar V$. The H\"older estimates applied to each function $u^\ep(\cdot,\omega)$ imply that this limit must hold uniformly in $\bar V$. 
\end{proof}

\subsection*{Acknowledgements}
SNA was partially supported by NSF Grant DMS-1004645 and by a Chaire Junior of la Fondation Sciences Math\'ematiques de Paris. CKS was partially supported by NSF Grant DMS-1004595.

\small
\bibliographystyle{plain}
\bibliography{graddep}

\end{document}